\documentclass[reqno]{amsart}

\usepackage{amsmath,amssymb,amscd, accents} \usepackage{graphicx}
\DeclareGraphicsExtensions{.eps} \usepackage{mathrsfs}
\usepackage[mathcal]{eucal}

\newtheorem{Thm}{Theorem}{\bfseries}{\itshape}
\newtheorem*{Thm*}{Theorem}{\bfseries}{\itshape}
\newtheorem{Cor}{Corollary}{\bfseries}{\itshape}
\newtheorem{Prop}[Cor]{Proposition}{\bfseries}{\itshape}
\newtheorem{Lem}[Cor]{Lemma}{\bfseries}{\itshape}
\newtheorem*{Lem*}{Lemma}{\bfseries}{\itshape}
{\bfseries}{\itshape}
{\bfseries}{\itshape}
\newtheorem{Def}[Cor]{Definition}{\bfseries}{\rmfamily}
{\scshape}{\rmfamily}
\newtheorem{Rem}[Cor]{Remark}{\scshape}{\rmfamily}
\newtheorem*{Claim}{Claim}{\bfseries}{\itshape}

\renewcommand\ge{\geqslant} \renewcommand\le{\leqslant}
\let\tildeaccent=\~ \let\hataccent=\^
\renewcommand\~[1]{\widetilde{#1}}

\def\<{\left<} \def\>{\right>} \def\({\left(} \def\){\right)}

\let\parasymbol=\S \def\secref#1{\parasymbol\ref{#1}}
 \def\pd#1#2{\tfrac{\partial#1}{\partial#2}}

\let\polishL=l \def\Zoladek.{\.Zol\c adek}

 \def\const{\operatorname{const}}
\def\codim{\operatorname{codim}}

 \def\ord{\operatorname{ord}}
 \def\etc.{\emph{etc}.}

\def\:{\colon} \def\R{{\mathbb R}} \def\C{{\mathbb C}} \def\Z{{\mathbb
    Z}} \def\N{{\mathbb N}} \def\Q{{\mathbb Q}} 
  \def\S{\varSigma}

 \let\PolishL=\L \def\Lojas.{\PolishL ojasiewicz}
\def\cN{{\mathcal N}} \def\cH{{\mathcal H}}
  
 \def\cL{{\mathcal L}} 
  
\def\cT{{\mathcal T}}

\def\cO{{\mathcal O}}

 \def\mult{\operatorname{mult}}

\def\rest#1{{\vert_{#1}}}

\def\vol{\operatorname{Vol}}
\def\ivol{\#}
\def\supp{\operatorname{supp}}

\def\qmi{W^s}

\def\spec{\operatorname{Spec}}

\def\conv{\Delta}
\def\qmi{W} 

\def\fm{{\mathfrak{m}}}
\def\vphi{\bar\phi}
\def\tc{\cT}
\def\hf{\cH}
\def\pdeg{{\deg_\Pi}}

\begin{document}

\title{Multiplicity estimates, analytic cycles and Newton polytopes}

\author{Gal Binyamini}\address{University of Toronto, Toronto, 
Canada}\email{galbin@gmail.com}
\thanks{The author was supported by the Banting
  Postdoctoral Fellowship and the Rothschild Fellowship}

\begin{abstract}
  We consider the problem of estimating the multiplicity of a
  polynomial when restricted to the smooth analytic trajectory of a
  (possibly singular) polynomial vector field at a given point or
  points, under an assumption known as the D-property. Nesterenko has
  developed an elimination theoretic approach to this problem which
  has been successfully applied to several problems in transcendental
  number theory.

  We propose an alternative approach to this problem using algebraic
  cycles and their local analytic structure. In particular we obtain
  simpler proofs to many of the best known estimates, and give more
  general formulations in terms of Newton polytopes, analogous to the
  Bernstein-Kushnirenko theorem. We also improve the estimate's
  dependence on the ambient dimension from doubly-exponential to an
  essentially optimal single-exponential.
\end{abstract}
\maketitle
\date{\today}

\section{Introduction}

Let $\xi$ be a polynomial vector field in $M=\C^n$ with
the coordinates $(x_1,\ldots,x_n)$,
\begin{equation} \label{eq:main-vf}
  \xi = \sum_{i=1}^n \xi_i(x) \pd{}{x_i}, \qquad \max_i \deg \xi_i = \delta.
\end{equation}
Let $p\in M$ and $\gamma_p$ a smooth holomorphic $\xi$-invariant curve
through $p$ which is not a subset of the singular locus of $\xi$. If
$\xi$ is non-singular at $p$ then $\gamma_p$ is the unique trajectory
of $\xi$ through $p$; otherwise $\gamma_p$ is a smooth analytic
separatrix of $\xi$ through $p$. Alternatively, using the coordinates
$(z,x_1,\ldots,x_n)$ on $M=\C^{n+1}$ we may consider a system of
(non-linear) polynomial differential equations
\begin{equation}\label{eq:diff-system}
  \frac{\partial x_i}{\partial z} = \frac{P_i(z,x_1,\ldots,x_n)}{Q_i(z,x_1,\ldots,x_n)},
  \quad \deg P_i,Q_i\le\delta, \qquad i=1,\ldots,n
\end{equation}
and their solution $f=(f_1,\ldots,f_n)$, viewed as a vector of
functions in the variable $z$ defined and holomorphic in some
neighborhood of $z_0\in\C$. We allow $p=(z_0,f(z_0))$ to be a singular
point of~\eqref{eq:diff-system} as long as the graph of $f$ in a
neighborhood of $z_0$ lies outside of the singular locus. With this
definition, the graph of $f$ forms a smooth separatrix $\gamma_p$ of
the vector field corresponding to~\eqref{eq:diff-system}.

Let $P\in\C[x_1,\ldots,x_n]$ be a polynomial with $\deg P=d$, and
suppose that $P\rest{\gamma_p}\not\equiv0$. We consider the following
question: \emph{can one give an upper bound for
  $\mult_p P\rest{\gamma_p}$ in terms of the parameter $d$?} More
specifically, can one give an explicit bound in terms of the
parameters $n,\delta,d$, or a bound involving existential constants
depending on $\gamma_p$? An answer to a question of this type is
referred to as a \emph{multiplicity estimate}.

\begin{Rem}
  Assuming that $\gamma_p$ is not contained in any proper algebraic
  set, it follows from basic linear algebra that for an appropriate
  choice of $P$ we have $\mult_p P\rest{\gamma_p}\ge\dim L(d)-1$,
  where $L(d)$ denotes the space of polynomials of degree bounded by
  $d$. Thus with respect to $d$, up to a multiplicative constant, the
  best possible multiplicity estimate is of order $d^n$.
\end{Rem}

More generally, one may replace the assumption $\deg P=d$ by a more
refined restriction, for instance assigning different degrees to each
variable. Two main cases have been considered in the literature:
estimates are given either in terms of a single degree $d$ as above,
or in terms of the two degrees $d_z=\deg_z P$ and $d_x=\deg_x P$. For
simplicity we will refer to the former as \emph{pure degree} and the
latter as \emph{mixed degree}, although to our knowledge only mixed
degrees of the specific type above have been considered in the
literature. In our approach the estimates for different types of
degrees are subsumed by a general estimate given in terms of the
Newton polytope of $P$, as explained in~\secref{sec:synopsis-newton}.
Arbitrary mixed degrees are obtained as a special case.

\subsection{The D-property and multiplicity estimates at singular points}
\label{sec:D-prop}

We now focus our attention on the case where $p$ is a singular point
of $\xi$. In this case, it is in general not possible to give a
multiplicity estimate depending only on $n,\delta,d$. For instance,
the linear field $x\partial_x+ay\partial_y$ with $a\in\N$ admits a
smooth trajectory $\gamma=\{y=x^a\}$ through the origin, and for $P=y$
we have $\mult_0 P\rest\gamma=a$. However, one may still hope that for a fixed
smooth analytic trajectory $\gamma_p$ it is possible to give a
good multiplicity estimate with respect to the degree $d$. Toward
this end, Nesterenko has introduced the following fundamental
definition.

\begin{Def}[\protect{\cite{nesterenko:galois,nesterenko:modular}}]
  Let $\gamma_p$ be a smooth analytic trajectory of $\xi$ through
  the point $p$. Then $\gamma_p$ is said to satisfy the \emph{D-property}
  (with constant $\chi$) if for any $\xi$-invariant variety $V\subset\C^n$,
  there exists a polynomial $P$ which vanishes identically
  on $V$ and satisfies $\mult_p P\rest{\gamma_p}\le\chi$.
\end{Def}

If $\xi$ is non-singular at $p$ and $\gamma_p$ is not contained in a
proper algebraic subset, then the D-property is automatically
satisfied with the constant $\chi=0$: there are no proper
$\xi$-invariant algebraic subsets containing $p$
(see~\secref{sec:algebraic-relations} for a discussion of the
situation where $\gamma_p$ is contained in a proper algebraic subset).
When $p$ is a singular point of $\xi$ the D-property is non-trivial.
See~\secref{sec:nest-d-prop} for a review of some systems satisfying
the D-property and their relation to transcendental number theory.

In the first part of the paper we present a new approach to the study
of multiplicity estimates for systems satisfying the D-property. In
particular we give new proofs for the multiplicity estimates presented
in~\secref{sec:nest-d-prop}, as well as their generalizations given in
\cite{dolgalev:mult}. We improve the dependence on the dimension $n$
from double to single exponential. Our result for the case of a single
point is as follows (for a more general result involving the sum of
multiplicities over multiple points see
Theorem~\ref{thm:mult-forest}).

\begin{Cor}\label{cor:mult-estimate}
  Let $p\in M$ and $\gamma_p$ a smooth analytic trajectory of $\xi$
  through $p$. Suppose that $\gamma_p$ satisfies the D-property with
  constant $\chi$. Then for any $P\in R$ with $\deg P=d$,
  \begin{equation}
    \mult_p^{\gamma_p} P \le (d+\tilde a_{n,\delta})^n + (2+\chi)(d+\tilde a_{n,\delta})^{n-1}
  \end{equation}
  where $\tilde a_{n,\delta}$ is the constant given in
  Theorem~\ref{thm:mult-forest}. In particular, it depends
  singly-exponentially on $n$ and polynomially on $\delta$.
\end{Cor}

In the second part of the paper we formulate and prove multiplicity
estimates for single or multiple points in terms of Newton polytopes
(see~\secref{sec:convex} for the notations). This includes the
mixed-degree estimates presented in~\secref{sec:nest-d-prop} as a
special case. Our result for the case of a single point is as follows
(for a more general result involving the sum of multiplicities over
multiple points see Theorem~\ref{thm:toric-mult-forest}; for a
slightly more refined estimate see
Corollary~\ref{cor:toric-mult-estimate}).

\begin{Thm}
  Let $p\in(\C^*)^n$ and $\gamma_p$ a smooth analytic trajectory of
  $\xi$ through $p$. Suppose that $\gamma_p$ satisfies the D-property
  with constant $\chi$. Then for any Laurent polynomial $P$ with
  $\Delta(P)=\Delta$,
  \begin{equation}
    \mult_p^{\gamma_p} P \le n!(3+\chi)\vol(\Delta+\Delta_{n,\xi})
  \end{equation}
  where $\Delta_{n,\xi}$ is some explicit polytope depending only on
  $n,\xi$. The diameter of $\Delta_{n,\xi}$ depends
  singly-exponentially on $n$.
\end{Thm}

We remark that the formulation in $(\C^*)^n$ rather than $\C^n$ is a
matter of elegance and technical convenience. If the polytope $\Delta$
is a convex co-ideal in $\Z^n_{\ge0}$ then by a simple translation
argument our estimate holds for any point $p\in\C^n$, as explained in
Remark~\ref{rem:co-ideals}.

\subsection{Synopsis of this paper}
\label{sec:synopsis}

Our principal contribution is an approach to multiplicity estimates
based on algebraic cycles, their local-analytic structure and
intersection theory. The local nature of our arguments simplifies the
proofs of many of the best known estimates, and extends naturally to
more general ambient spaces (as illustrated by our estimates in terms
of Newton polytopes, related to toric compactifications). This
approach also admits a natural generalization to the study of
multiplity estimates for foliations of dimension greater than one ---
this falls outside the scope of the present paper, but
see~\secref{sec:conclusion} for a brief discussion.

In~\secref{sec:synopsis-pure} we give an outline of our approach in
the pure degree case. In~\secref{sec:synopsis-newton} we discuss the
generalization of these multiplicity estimates to the context of the
theory of Newton polytopes. In~\secref{sec:synopsis-organization} we
describe the organization of the paper.

\subsubsection{The pure degree case}
\label{sec:synopsis-pure}

We begin by describing our approach in the case of pure degrees. A
general paradigm for proving multiplicity estimates, which has been
used by different authors in various ways, is as follows:
\begin{itemize}
\item Associate to each ideal a \emph{multiplicity}, such that the
  multiplicity of the principal ideal generated by $P$ essentially
  agrees with $\mult_p P\rest{\gamma_p}$.
\item Show how to construct from a given ideal $I$ a new ideal $J$, of
  smaller dimension (using the derivative operator $\xi$), such that
  the multiplicity of $I$ is bounded in terms of the multiplicity of
  $J$.
\item Show that one eventually obtains an ideal for which the multiplicity
  is known, for example the whole ring.
\end{itemize}
For instance, in \cite{gabrielov:mult-old,ny:chains} this paradigm was
used, where the multiplicity associated to $I$ is
$\min_{F\in I}[\mult_p F\rest{\gamma_p}]$. In Nesterenko's approach a
more refined notion of multiplicity for unmixed ideals was used
(see~\secref{sec:nest-approach} for details).

Our approach follows the same paradigm, working with algebraic cycles
in place of ideals. We introduce a local analytic notion of the
\emph{multiplicity of an analytic cycle along $\gamma_p$} (see
Definition~\ref{def:cycle-curve-mult}). Namely, for any germ of an
irreducible analytic variety $V\subset M$ at $p$, we define the
multiplicity of $V$ along $\gamma_p$, denoted $\mult_p^{\gamma_p} V$
to be the Samuel multiplicity of the ideal of functions vanishing on
$\gamma_p$, restricted to the local ring $\cO_{V,p}$ of germs of
regular functions on $V$ at the point $p$. We extend this notion, by
linearity, to arbitrary analytic cycles.

Let $V$ be a germ of an analytic set at $p$ and $f$ the germ of an
analytic function with $f\rest V\equiv0$ and $g=\xi f$ with
$g\rest V\not\equiv0$. Let $\Gamma$ denote the intersection cycle
$V\cdot V(g)$. We prove the estimate
$\mult_p^{\gamma_p}V\le\mult_p V+\mult_p^{\gamma_p}\Gamma$ where
$\mult_p V$ denotes the multiplicity of the analytic germ $V$ at $p$
(see Lemma~\ref{lem:rolle-cycle}). This may be viewed as a local form
of Rolle's lemma, relating the number of zeros of a function (or in
the local case, the multiplicity) to the zero locus of its derivative.

We recursively define a forest (union of trees) where each node is an
irreducible variety with an associated multiplicity, as follows:
\begin{enumerate}
\item The roots are given by the components of $V(P)$ with their
  multiplicities.
\item If $n[V]$ is a node and $V$ is a point, or is contained in a
  proper $\xi$-invariant variety, then this node is a leaf.
\item Otherwise, we choose a polynomial $F$ vanishing on $V$ with
  $G=\xi F$ and $G\rest V\not\equiv0$ (as explained below), and let
  the children of $n[V]$ be the components of the intersection cycle
  $n[V]\cdot V(G)$.
\end{enumerate}

Assume now that $n[V]$ is a node of the forest above, which is not
contained in any proper $\xi$-invariant variety. In this case we show
that one can always choose a polynomial $G$ as in item~(3) above, with
$\deg G\le \deg P+\tilde a_{n,\delta}$ where $a_{n,\delta}$ is some universal
constant depending only on $n,\delta$ and growing singly-exponentially
with $n$ (see Lemma~\ref{lem:main} and Theorem~\ref{thm:mult-forest}).

This step is similar to a lemma appearing in the work of Nesterenko
(see~\secref{sec:nest-approach} for details). However, our approach to
the proof is different, relying on local analytic considerations
concerning the multiplicities of analytic germs. This allows us to
give a shorter and more transparent proof, and improve the growth of
the constants from doubly-exponential in Nesterenko's theorem to
single-exponential. Later this also allows us to relatively easily
extend our arguments from the pure degree case to general Newton
polytopes.

We are now ready to complete the argument. Suppose that $\gamma_p$
satisfies the D-property with constant $\chi$. We show that the
multiplicity $\mult_pP\rest{\gamma_p}$ is bounded by the sum of the
(local analytic) multiplicities of all the nodes in the forest above
at the point $p$, where the multiplicities of the leafs are taken with
coefficient $\chi$ (except for isolated points, which may be taken
with the coefficient 1). Our proof proceeds by a simple tree induction
over the forest, using the following properties for the roots, the
recursive step, and the leafs respectively.
\begin{itemize}
\item $\mult_p P\rest{\gamma_p}=\mult_p^{\gamma_p} V(P)$ (see Proposition~\ref{prop:divisor-mult}).
\item By the Rolle-type theorem above, if $\Gamma$ is a node and
  $\Gamma'$ denotes the sum of its children, then
  $\mult_p^{\gamma_p}\Gamma\le\mult_p\Gamma+\mult_p^{\gamma_p}\Gamma'$.
\item If $V$ is a point then $\mult_p^{\gamma_p}V=1$, and if $V$ is
  contained in a proper $\xi$-invariant variety then
  $\mult_p^{\gamma_p}V\le\chi\mult_p V$ (see Proposition~\ref{prop:cycle-dist-vs-mult}).
\end{itemize}

The proof of the multiplicity estimate is concluded in a
straightforward manner by computing an upper bound for the degrees of
all nodes appearing in the forest using the Bezout theorem,
thereby in particular bounding their multiplicities at $p$.

The forest constructed in our proof describes the behavior of the
multiplicity function not only at a single point $p$, but in fact over
any collection of points satisfying the D-property.
In~\secref{sec:recover-nesterenko} we show how to derive the
multiplicity estimates of \cite{nesterenko:modular,dolgalev:mult} for
the case of multiple points from the geometry of this forest.

\subsubsection{The toric case: estimates in terms of Newton polytopes}
\label{sec:synopsis-newton}

For some applications it is not sufficient to give the multiplicity
estimate in terms of a pure degree $d$. For instance, in
Nesterenko's Theorem~\ref{thm:nest-main} and its intended application
it is important to describe the dependence of the multiplicity
function on the $z$-degree $d_z$ and the $x$-degree $d_x$ separately.
Estimates in terms of the pure degree $d$ are natually tied to the
study of varieties in the projective space $\C P^n$, referred to in
the literature as the \emph{absolute} case. Estimates in the mixed
degree case, i.e. for the two separate degrees $d_z,d_x$ are naturally
tied to the study of varieties in the product $\C P^1\times\C P^n$,
referred to in the literature as the \emph{relative} case.

It is natural to expect that similar types of estimates should hold
for different notions of degree. Rather than consider the different
possible compactifications of the affine space corresponding to each
type of degree, we present a uniform expression of a multiplicity
estimate in the framework of Newton polytope theory. Namely, by the
Bernstein-Kusnirenko theorem (see~\secref{sec:bk}) it is natural to
expect that for a polynomial $P$ with Newton polytope $\Delta$, the
multiplicity $\mult_p^{\gamma_p} P$ would be essentially bounded by a
constant times the volume $\vol(\Delta)$. A result of this type would
immediately generalize the pure and mixed degree cases: the former
corresponds to $\Delta=d\Delta_x$ and the latter to
$\Delta=d_z\Delta_z+d_x\Delta_x$, where $\Delta_z,\Delta_x$ denote the
standard polytopes in the $z$ and $x$ variables respectively.

In~\secref{sec:toric-estimates} we repeat our proof in the context of
Newton polytope theory, and obtain an analogous description of the
multiplicity function in terms of a forest of cycles. This leads to
similar multiplicity estimates over one or several points in terms of
the volumes of the Newton polytopes associated to $P$ and the vector
field $\xi$.

\subsubsection{Organization of this paper}
\label{sec:synopsis-organization}
  
In~\secref{sec:cycles} we review the basic notions related to cycles
and their intersections; define the notion of the multiplicity of an
analytic cycle along a smooth analytic curve; and prove the basic
results concerning this notion. In~\secref{sec:pure-degree} we prove
our multiplicity estimates in the most familiar pure degree case. All
the key ideas are already present in this context.
In~\secref{sec:convex} we give some background on the theory of Newton
polytopes including the Bernstein-Kushnirenko theorem, and prove some
elementary results in convex geometry that are needed in the sequel.
In~\secref{sec:toric-estimates} we prove our multiplicity estimates
for general Newton polytopes, and show how these results imply
Theorem~\ref{thm:nest-main} and its various generalizations.
In~\secref{sec:conclusion} we list some general conluding remarks. In
the appendix we give a brief historical review of various multiplicity
estimates and their origins; describe Nesterenko's principal results
on the D-property and multiplicity estimates; and sketch the proof of
Nesterenko's estimates with appropriate references to the present
paper for comparison.

\section{Cycles and multiplicities}
\label{sec:cycles}

For simplicity we present the results of this section in the context
of the ambient space $M=\spec R$, where $R$ is either
$\C[x_1,\ldots,x_n],\C[x_1^\pm,\ldots,x_n^\pm]$, or the ring $\cO_p$
of germs of holomorphic functions in $(\C^n,p)$. However, we note that
many of the results could be extended without change to more general
ambient spaces.

\subsection{Cycles and their intersections}

We give a brief review of the basic notions related to cycles and
their intersection product. For a canonical reference see
\cite{fulton:it}.

Recall that a $k$-cycle is defined to be a formal sum $\sum n_i [V_i]$
where $n_i\in\Z$ and $V_i\subset M$ are irreducible varieties of
dimension $k$. A general \emph{cycle} is a sum of cycles of various
dimensions. In this paper we shall deal only with cycles with positive
coefficients.

Two irreducible varieties $V,W\subset M$ are said to intersect
\emph{properly} at a component $Z\subset V\cap W$ if
$\codim Z=\codim V+\codim W$. In this case we have a well defined
number $i(Z;V\cdot W;M)$, the \emph{multiplicity} of $Z$ in
the intersection of $V$ and $W$. If every component of the
intersection $V\cap W$ is proper then one has a well defined
intersection product
\begin{equation}
  V\cdot W = \sum_{Z\subset V\cap W} i(Z;V\cdot W;M) [Z]
\end{equation}
This can be extended by linearity to the product of arbitrary
cycles, assuming every pair of components in the product intersect
properly. This product is commutative and associative.

To each function $f\in R$ one can associate its divisor $V(f)$,
which is an $n-1$-cycle. In particular, if $\gamma\subset M$ is
a curve passing through $p$ then
\begin{equation}\label{eq:vf-mult}
  i(p; V(f)\cdot\gamma;M) = \mult_p (f\rest\gamma)
\end{equation}
where $\mult$ above denotes the usual multiplicity of a holomorphic
function on a holomorphic curve.

If $V\subset M$ is an irreducible variety, we denote by $\cO_V$ the
ring of regular functions on $V$ (i.e. $\cO_V:=R/I_V$ where $I_V$
denotes the ideal of functions vanishing on $V$). If $p\in M$, we
denote by $\cO_{V,p}$ the corresponding ring of germs of holomorphic
functions on $V$ at the point $p$. Note that we consider holomorphic
localization even if $V$ is defined in the algebraic category.

Recall that the multiplicity of $V$ at a point $p\in M$, denoted
$\mult_p V$, is defined to be $e(\fm_p,\cO_{V,p})$ where $\fm_p$ denotes
the maximal ideal of $\cO_{V,p}$. Geometrically, we
have
\begin{equation} \label{eq:mult-def}
  \mult_p V = i(p;V\cdot L;M)
\end{equation}
where $L$ denotes a generic affine linear plane of codimension $\dim V$
passing through $p$. We extend this definition by linearity to arbitrary
cycles.

We record the following standard fact.

\begin{Prop}\label{prop:mult-nodrop}
  Let $\Gamma$ denote a cycle and $f\in R$, and suppose that $\Gamma$
  intersects $V(f)$ properly. Then for any $p\in V(f)$ we have
  \begin{equation}
    \mult_p \Gamma \le \mult_p \Gamma\cdot V(f)
  \end{equation}
\end{Prop}
\begin{proof}
  By linearity it suffices to prove the inequality for a single
  irreducible variety $V$ of dimension $k$. If $p\notin V$ there
  is nothing to prove. Otherwise, the left hand side is given by
  $e(\fm_p,\cO_{V,p})$ and the right hand side is given by
  $e(\<f,\ell_1,\ldots,\ell_{k-1}\>,\cO_{V,p})$ where
  $\ell_1,\ldots,\ell_{k-1}$ denote $k-1$ generic linear functionals.
  The claim follows since the Samuel multiplicity is monotonic with
  respect to inclusion of ideals.

  Alternatively, the reader may prove this claim by arguing in the
  same manner as in the proof of Lemma~\ref{lem:rolle-cycle}.
\end{proof}

If $\Gamma=\sum n_i [p_i]$ is a zero dimensional cycle, we define its
degree to be $\deg\Gamma:=\sum n_i$. If $M$ is $\C^n$ or $(\C^*)^n$ we
extend this to cycles of arbitrary dimension by defining
$\deg\Gamma:=\deg(\Gamma\cdot L)$ where $L$ is a generic affine linear
space of dimension complementary to $\Gamma$.

\subsection{Multiplicity of a cycle along a curve}

We now define the multiplicity of an irreducible variety through a
smooth analytic curve.

\begin{Def} \label{def:cycle-curve-mult}
  Let $V\subset M$ be an irreducible variety and $\gamma\subset M$ a
  smooth analytic curve passing through a point $p\in V$, and assume
  that $\gamma\not\subset V$. Let $I_\gamma\subset\cO_{V,p}$ denote
  the restriction to $\cO_{V,p}$ of the ideal of functions vanishing
  on $\gamma$. Then $I_\gamma$ is $\fm_p$-primary, and we define the
  \emph{multiplicity of $V$ through $\gamma$ at $p$}, denoted
  $\mult_p^\gamma V$, to be $e(I_\gamma,\cO_{V,p})$. If
  $\gamma\subset V$ we define $\mult^\gamma_p V=\infty$.

  We extend this definition by linearity to arbitrary cycles.
\end{Def}

The preceding definition admits a simple geometric interpretation
similar to~\eqref{eq:mult-def}. Indeed, since $\gamma$ is smooth, we
may choose analytic coordinates under which $\gamma$ is linear. In
this case, we have
\begin{equation}\label{eq:cycle-mult-def}
  \mult_p^\gamma V = i(p;V\cdot L; M)
\end{equation}
where $L$ denotes a generic affine linear plane of codimension
$\dim V$ \emph{containing $\gamma$}.

The following proposition shows that the multiplicity of a function
along a curve $\gamma$ can be interpreted in terms of the
corresponding cycle.

\begin{Prop} \label{prop:divisor-mult}
  Let $f\in R$ and let $\gamma$ be a smooth analytic curve,
  $f\rest\gamma\not\equiv0$. Then
  \begin{equation}
    \mult_p^\gamma V(f) = \mult_p (f\rest\gamma)
  \end{equation}
\end{Prop}
\begin{proof}
  Choose analytic coordinates making $\gamma$ linear. Then since
  $\codim\gamma=n-1=\dim V$, we have by~\eqref{eq:cycle-mult-def}
  \begin{equation}
    \mult_p^\gamma V(f) = i(p; V(f)\cdot\gamma;M)
  \end{equation}
  and the proposition follows by~\eqref{eq:vf-mult}.
\end{proof}

\subsection{Multiplicity of a cycle along a vector field}

We now define the multiplicity of a cycle with respect to a vector
field. We say that a vector field $\xi$ is an $R$-vector field
if its coefficients in the standard coordinates are functions
from $R$.

\begin{Def}
  Let $\Gamma\subset M$ be a cycle and $\xi$ be an $R$-vector field.
  Let $p\in M$ be a nonsingular point of $\xi$, and let $\gamma_p$
  denote the trajectory of $\xi$ through $p$. We define the
  \emph{multiplicity of $\Gamma$ through $\xi$ at $p$} to be
  $\mult^\xi_p\Gamma :=\mult^{\gamma_p}_p \Gamma$.
\end{Def}

Note that when a vector field is singular at a point $p$ and admits a
smooth analytic trajectory $\gamma_p$ through $p$, we may still talk
about the multiplicity of a cycle through the curve $\gamma_p$. We
prove a Rolle-type lemma for cycles (cf.
\cite[Lemma~5.1]{nesterenko:modular}). It applies for singular as well
as nonsingular points of a vector field.

\begin{Lem}\label{lem:rolle-cycle}
  Let $V\subset M$ be an irreducible variety of dimension $k$. Let
  $\xi$ denote an $R$-vector field defined (possibly singular) near
  $p$, and let $\gamma\subset M$ be a smooth analytic trajectory of
  $\xi$ passing through $p$.

  Suppose that $f\in I_V$ and $g=\xi f\notin I_V$. Then
  \begin{equation}
    \mult_p^\gamma V \le \mult_p V + \mult_p^\gamma(V\cdot V(g)).
  \end{equation}
  Moreover, if $\xi$ is singular at $p$ we may omit the $\mult_p V$
  term,
  \begin{equation}
    \mult_p^\gamma V \le \mult_p^\gamma(V\cdot V(g)).
  \end{equation}
\end{Lem}
\begin{proof}
  Choose analytic coordinates making $\gamma$ linear. If $L^{k-1}$ is a
  sufficiently generic affine plane of codimension $k-1$ and $L$ is
  a sufficiently generic affine hyperplane, both containing $\gamma$,
  then we have by~\eqref{eq:cycle-mult-def}
  \begin{gather}
    \mult_p^\gamma V = i(p; V\cdot(L^{k-1}\cdot L);M) = i(p;C\cdot L;M) \\
    \mult_p^\gamma (V\cdot V(g)) = i(p;(V\cdot V(g))\cdot L^{k-1};M) = i(p;C\cdot V(g);M)
  \end{gather}
  where $C$ denotes the curve $V\cdot L^{k-1}$. In deriving this we
  have used the associativity and commutativity of the intersection
  product, and the fact that all intersections above are proper for
  sufficiently generic $L^{k-1},L$. Note that since $L^{k-1}$ is
  chosen generically from a linear system with base locus $\gamma$, we
  may in fact assume (by Bertini's theorem) that it is a reduced curve
  without assigned multiplicities (although this does not play a role
  in our arguments).

  The intersection of $L^{k-1}$ with a generic hyperplane is a generic
  plane of codimension $k$. This implies that $\mult_p V = \mult_p C$,
  and we denote this number by $\mu$. Let $\ell$ denote the
  affine-linear function with $L=V(\ell)$. Our statement is thus
  reduced to
  \begin{equation}
    i(p; C\cdot V(\ell);M) \le \mu + i(p;C\cdot V(g); M)
  \end{equation}
  and similarly, without the $\mu$ term, for the case when $\xi$ is
  singular at $p$.

  Making a linear change of coordinates we may assume that the
  coordinates are given by $(x_1,\ldots,x_n)$ where $p$ corresponds to
  the origin, $x_1$ is transversal to $C$ and $\gamma$ at the origin,
  and $\gamma$ is given by the vanishing of $x_2,\ldots,x_n$. Then $C$
  admits $\mu$ real pro-branches
  \begin{equation}
    C_i = \{(t,\vphi^i(t)):=(t,\phi^i_2(t),\ldots,\phi^i_n(t)) : t\in\Q_{\ge0}\} \qquad i=1,\ldots,\mu
  \end{equation}
  where the $\phi^i_j(t)$ admit Puiseux expansions and
  $\ord_0\phi^i_j(t)\ge1$. By~\eqref{eq:vf-mult} and a well-known
  formula for the multiplicity of a function on a curve, we have for
  every analytic function $h$ that
  \begin{equation}
    i(p;C\cdot V(h);M)=\mult_p (h\rest C) = \sum_{i=1}^\mu \ord_0 h(t,\vphi^i(t)).
  \end{equation}
  Thus the claim will be proved once we show that for $i=1,\ldots,\mu$,
  \begin{equation}\label{eq:ord-compare}
    \ord_0 \ell(t,\vphi^i(t)) \le 1+\ord_0 g(t,\vphi^i(t))
  \end{equation}
  and similarly, without the $1$ term, for the case when $\xi$ is
  singular at $p$.

  Let $v(\vphi^i):=\min_{j=2,\ldots,n}\ord_0\phi^i_j(t)$, essentially
  measuring the asymptotic distance between $\vphi^i$ and $\bar0$ in
  powers of $t$. Since $\ell$ is a generic linear combination of the
  $x_2,\ldots,x_n$ coordinates, the left hand side
  of~\eqref{eq:ord-compare} is equal to $v(\vphi^i)$ (one only needs
  to make the choice generic enough to avoid cancellation between the
  leading terms of $\phi^i_j, j=2,\ldots,n$). On the other hand,
  $f(t,\vphi^i(t))\equiv0$ by assumption, and hence
  $\ord_0 f(t,\bar0)\ge v(\vphi^i)$. But the $x_1$ axis is a
  trajectory of $\xi$, and since derivation cannot decrease the order
  of an analytic function by more than $1$ we have
  \begin{equation}
    \ord_0g(t,\bar0) = \ord_0 (\xi f)(t,\bar0) \ge v(\vphi^i)-1.
  \end{equation}
  Finally, translating back to $C_i$, we have
  $\ord_0 g(t,\vphi^i(t))\ge v(\vphi^i)-1$. In the singular case
  derivation by $\xi$ does not decrease the order of $f$, and we
  obtain a similar result without the $1$ term, as claimed.
\end{proof}

We also have the following upper bound.

\begin{Prop} \label{prop:cycle-dist-vs-mult}
  Let $V\subset M$ be an irreducible variety and let $\gamma\subset M$
  be a smooth analytic curve passing through $p$.

  Suppose that $f\in I_V$ and $\mult_p f\rest\gamma=\delta$. Then
  \begin{equation}
    \mult^\gamma_p V \le \delta \cdot \mult_p V
  \end{equation}
\end{Prop}
\begin{proof}
  We keep the notations from the proof of Lemma~\ref{lem:rolle-cycle}
  and argue in a similar manner. By the same arguments, it will
  suffice to prove that $v(\vphi^i)\le\delta$ for $i=1,\ldots,\mu$.
  Assuming the contrary, we see that $f(t,\vphi^i(t))\equiv0$ implies
  $\ord_0f(t,\bar0)>\delta$ contrary to the conditions of the
  proposition.
\end{proof}

\subsection{Multiplicities at generic points}

Let $W\subset M$ be an irreducible variety. We introduce the following
notation to simplify our exposition. Let $\xi$ be an $R$-vector field,
$\Gamma$ a cycle and $f\in R$. For each of the multiplicity functions
$\mult_p\Gamma$, $\mult^\xi_p f$ and $\mult^\xi_p\Gamma$ we define the
multiplicity functions $\mult_W\Gamma$, $\mult^\xi_W f$ and
$\mult^\xi_W\Gamma$ to denote the value of the corresponding
multiplicity function at a generic point $p\in W$.

It is easy to verify that the definition above is well-defined, i.e.
that for $p$ outside a set $\Sigma$ of positive codimension in $W$ the
multiplicity functions above are constant (and take larger values on
$\Sigma$). It is also easy to verify that the two multiplicity
functions involving $\Gamma$ above are linear with respect to
$\Gamma$.

The following is a simple transversality statement.

\begin{Prop}\label{prop:mult-transverse}
  Let $V\subset M$ be an irreducible variety and $\xi$ an $R$-vector field.
  Suppose that $V$ is not invariant under $\xi$. Then $\mult^\xi_V V=1$.
\end{Prop}
\begin{proof}
  The claim follows since, at a generic point of $V$, $\xi$ must be
  transversal to $V$. Formally we may argue that by assumption $I_V$
  is not a $\xi$-invariant ideal, and applying
  Lemma~\ref{lem:rolle-cycle} with any polynomial $P\in I_V$ such that
  $\xi P\notin I_V$ proves the claim, since $\mult_V V=1$ (as $V$ is
  smooth at a generic point).
\end{proof}

\section{The multiplicity estimate in the pure degree case}
\label{sec:pure-degree}

We now restrict attention to the case $M=\C^n$. We let $\xi$ denote
the vector field
\begin{equation}
  \xi = \sum_{i=1}^n \xi_i(x) \pd{}{x_i}, \qquad \max_i \deg \xi_i = \delta.
\end{equation}

Let $\cN(n,\delta,d)$ denote the maximal possible (finite) value of
$\mult_p^\xi P$ for any $\xi$ as above, $p\in M$ a nonsingular point of $\xi$,
and $P\in R$ with $\deg P\le d$. In \cite{mult-morse} the following is
proved.

\begin{Thm} \label{thm:mult-morse}
  With the notation above,
  \begin{equation}
    \cN(n,\delta,d) \le 2^{n+1} (d+(n-1)\delta)^n
  \end{equation}
\end{Thm}

We recall the standard notion of Hilbert functions. For our purposes
it is more convenient to work in a given affine chart
$x_1,\ldots,x_n$. We define for an affine variety $V\subset\C^n$,
\begin{equation}
  H(V,t) = \dim_C L(V,t), \qquad L(V,t):=\{P\rest V : P\in R, \deg P\le t\}
\end{equation}

Estimates of the following type were given, in projective form, in
\cite{nesterenko:hf} (see also \cite{bertrand:hf}). The affine version
follows readily from the corresponding projective estimate in
homogeneous variables $x_0,\ldots,x_n$ by restricting to the chart
$(1:x_1,\ldots,x_n)$. The reader may also consult the proof of
Proposition~\ref{prop:toric-hf-bound}.

\begin{Prop}\label{prop:hf-est-pd}
  Let $V\subset\C^n$ be an affine variety of dimension $k$. Then
  \begin{equation}
    \hf(V,t) \le \deg(V)\cdot t^k + k.
  \end{equation}
\end{Prop}

The following lemma, a simple corollary of
Proposition~\ref{prop:hf-est-pd}, shows that the ideal of a variety of
sufficiently small degree must contain polynomials of bounded degree.

\begin{Lem} \label{lem:hf-poly-finder}
  Let $d>2n$ and let $V$ be an irreducible variety of dimension $k$ with
  \begin{equation}
    \deg V \le A_n^{-1} d^{n-k} \qquad A_n= 2 n!
  \end{equation}
  Then $I_V$ contains a non-zero polynomial $P$ with $\deg P<d$.
\end{Lem}
\begin{proof}
  By Proposition~\ref{prop:hf-est-pd},
  \begin{equation}
    \hf(V,d-1) \le \deg(V) (d-1)^k  + k \le A_n^{-1} d^n + k \le 2A_n^{-1} d^n
  \end{equation}
  On the other hand,
  \begin{equation}
    \hf(\C^n,d-1) = \binom{d+n-1}{n} > d^n / n!
  \end{equation}
  It follows that some non-zero polynomial of degree bounded by
  $d-1$ vanishes on $V$, as claimed.
\end{proof}

The following lemma plays a key role in our arguments (cf.
\cite[Lemma~5.4]{nesterenko:modular}, see~\secref{sec:nest-approach} for
discussion).

\begin{Lem} \label{lem:main} Let $V\subset M$ be an irreducible
  variety of dimension $k$ and suppose that $V$ is not contained in a
  (non-trivial) $\xi$-invariant variety. Let $P$ be a non-zero
  polynomial of minimal degree in $I_V$. Then
  \begin{equation}
    \mult^\xi_V P\le a_{n,\delta}, \qquad a_{n,\delta}=\cN(n,\delta,A_n2^nn\delta)
  \end{equation}
\end{Lem}
\begin{proof}
  We first note that $\mult^\xi_V P$ is finite. Indeed, otherwise $V$
  would be contained in the invariant variety defined by $\<P,\xi P,\ldots\>$
  contrary to the conditions of the lemma.

  Let $d:=\deg P$. Henceforth we assume that $d>A_n 2^n n \delta$.
  Otherwise, the statement of the lemma follows from
  Theorem~\ref{thm:mult-morse}.

  \begin{Claim}
    Let $W\subset M$ be an irreducible variety of dimension $l\ge k$
    with $V\subset W$. Then $\deg W > A_n^{-1} d^{n-l}$.
  \end{Claim}
  \begin{proof}
    Indeed, otherwise by Lemma~\ref{lem:hf-poly-finder} we have a polynomial
    of degree smaller than $d$ in $I_W$, and since $I_W\subset I_V$ this contradicts
    the minimality of $d=\deg P$.
  \end{proof}
  
  We now proceed with the proof of the lemma. We will construct a
  sequence of cycles $V(P)=\Gamma^1\supset\cdots\supset\Gamma^{n-k}$ with
  the following properties:
  \begin{enumerate}
  \item $\codim \Gamma^j = j$.
  \item Every component of $\Gamma^j$ contains $V$.
  \item For every irreducible variety $W$ with
    $V\subset W\subset\supp\Gamma^j$ we have
    \begin{equation}
      \mult^\xi_WP \le (j-1)\mult_W\Gamma^j+\mult^\xi_W\Gamma^j
    \end{equation}
  \item The degree of $\Gamma^j$ is bounded, $\deg\Gamma^j \le (2d)^j$.
  \end{enumerate}
  The first cycle $\Gamma^1:=V(P)$ clearly satisfies the conditions
  (condition~(3) follows from Proposition~\ref{prop:divisor-mult} for
  any point $p$, and certainly for generic points in $W$).

  We proceed with the construction by induction. Suppose that
  $\Gamma^j$ has been constructed, $\Gamma^j=\sum m^j_i [W^j_i]$. By
  assumption we have $V\subset W^j_i$, and by the preceding claim
  $\deg W^j_i> A_n^{-1} d^j$. Combinig this with condition~(4) we
  see that $m^j_i< A_n 2^j$. By condition~(3) and
  Proposition~\ref{prop:mult-transverse} we have
  \begin{equation} \label{eq:p-wji-mult}
    \begin{aligned}
      \mult^\xi_{W^j_i} P &\le (j-1)\mult_{W^j_i} \Gamma^j + \mult^\xi_{W^j_i}\Gamma^j \\
       &= (j-1)m^j_i + m^j_i < A_n 2^j j
    \end{aligned}
  \end{equation}
  If $\mult^\xi_{W^j_i}P=\mult^\xi_VP$ then~\eqref{eq:p-wji-mult}
  proves the claim of the lemma. Otherwise, let $Q^j_i$ denote the
  first $\xi$-derivative of $P$ which does not vanish on $W^j_i$ (but
  still vanishes on $V$). By~\eqref{eq:p-wji-mult} we have
  \begin{equation}\label{eq:qij-deg}
    \deg Q^j_i\le d+(A_n 2^j j)\delta \le 2 d
  \end{equation}

  Let $\tilde\Gamma^{j+1}_i=[W^j_i]\cdot V(Q^j_i)$ and
  $\Gamma^{j+1}_i$ consist of the components of $\tilde\Gamma^{j+1}_i$
  which contain $V$. Finally, define
  $\Gamma^{j+1}=\sum m^j_i \Gamma^{j+1}_i$. Conditions~(1) and~(2)
  hold by definition, and condition~(4) follows inductively by
  linearity and the Bezout theorem.

  We move now to the proof of condition~(3). We have for any
  $V\subset W\subset \supp \Gamma^{j+1}_i$ the inequality
  \begin{equation}\label{eq:wji-mult}
    \mult_W W^j_i \le \mult_W \tilde\Gamma^{j+1}_i = \mult_W \Gamma^{j+1}_i
  \end{equation}
  where we used Proposition~\ref{prop:mult-nodrop} for generic points
  in $W$, and the fact that the components of
  $\tilde\Gamma^{j+1}_i-\Gamma^{j+1}_i$ do not contain $V$ (and
  certainly do not meet generic points of $W$). Similarly, using
  Lemma~\ref{lem:rolle-cycle} for generic points in $W$ we have
  \begin{equation}\label{eq:wji-xi-mult}
    \begin{aligned}
      \mult^\xi_W W^j_i &\le \mult_W W^j_i + \mult^\xi_W \tilde\Gamma^{j+1}_i \\
      &\le \mult_W \Gamma^{j+1}_i + \mult^\xi_W \Gamma^{j+1}_i
    \end{aligned}
  \end{equation}
  Condition~(3) now follows inductively from~\eqref{eq:wji-mult}
  and~\eqref{eq:wji-xi-mult} and linearity of the multiplicity
  function. This concludes our construction.

  Conditions~(1) and~(2) guarantee that $\Gamma^{n-k}$ is supported on
  $V$, and in the notation above $W^{n-k}_1=V$.
  Thus~\eqref{eq:p-wji-mult} implies the claim of the lemma.
\end{proof}

\begin{Rem}
  We have not attempted to optimize the dependence on all parameters
  in the proof above, which could clearly be improved at a number of
  points in the argument. Our main goal (as far as explicit estimates
  are concerned) was to provide a bound admitting single-exponential
  growth with $n$. Note that one can repeat the proof above without
  relying on Theorem~\ref{thm:mult-morse}, but the estimates in this
  case become significantly worse.
\end{Rem}

\subsection{The multiplicity forest}

We start with a definition.
\begin{Def}
  A \emph{cycle forest} $T$ is a directed forest (union of directed
  trees) whose nodes are irreducible varieties with assigned
  multiplicities, such that:
  \begin{enumerate}
  \item All nodes of level $k$ have codimension $k$.
  \item All children of a node $n[V]$ are subvarieties of $V$.
  \end{enumerate}
  We will denote by $T^k$ the cycle formed by the sum of all nodes
  of level $k$ (with their assigned multiplicities), and by $\cL(T)$
  the set of leafs of $T$.
\end{Def}

Our principal result it the following theorem.
\begin{Thm}\label{thm:mult-forest}
  Let $P\in R$ with $\deg P=d$. There exists a cycle forest $T_P$ with
  the following properties.
  \begin{enumerate}
  \item The roots are given by the components of $V(P)$ (with their
    assigned multiplicities).
  \item Every leaf of $T_P$ is either an isolated point, or a
    variety contained in some $\xi$-invariant variety.
  \item The degree of $T_P^k$ is bounded,
    \begin{equation}
      \deg T_P^k \le (d+\tilde a_{n,\delta})^k \qquad \tilde a_{n,\delta}=\delta a_{n,\delta}
    \end{equation}
  \item Denote by $\cL^+(T)$ (resp. $\cL^0(T)$) the set of leafs of
    $T$ that have positive dimension (resp. dimension zero). If
    $\gamma_p$ is a smooth analytic trajectory of $\xi$ through $p$
    which satisfies the D-property at $p$ with constant $\chi$, then
    \begin{equation}
      \mult_p^{\gamma_p} P \le \sum_{\Gamma\in T\setminus\cL^+(T_P)} \mult_p\Gamma
      + \chi \sum_{\Gamma\in\cL^+(T_P)}\mult_p\Gamma
    \end{equation}
    If $\xi$ is singular at $p$ this bound may be tightened to
    \begin{equation}
      \mult_p^{\gamma_p} P \le \sum_{\Gamma\in \cL^0(T_P)} \mult_p\Gamma
      + \chi \sum_{\Gamma\in\cL^+(T_P)}\mult_p\Gamma
    \end{equation}
  \end{enumerate}
\end{Thm}
\begin{proof}
  We construct $T$ recursively, starting with the roots specified in
  condition~(1). Suppose $n[V]$ is a node. If $V$ is a point, or is
  contained in an invariant variety, then this node is a leaf.
  Otherwise, let $\tilde P$ be a polynomial of minimal degree in
  $I_V$. Since $V\subset V(P)$ by definition of a cycle forest, we
  have $P\in I_V$, and hence $\deg \tilde P\le d$. By
  Lemma~\ref{lem:main}, $\mult^\xi_V\tilde P\le a_{n,\delta}$. Let $Q$
  denote the first derivative of $\tilde P$ which does not vanish
  identically on $V$. Then $\deg Q\le d+\tilde a_{n,\delta}$. We
  define the children of $n[V]$ to be the components of
  $n[V]\cdot V(Q)$ (with their assigned multiplicities).

  Conditions~(1) and~(2) hold by definition. Condition~(3) follows
  inductively by application of the Bezout theorem. Finally,
  condition~(4) follows by a straightforward recursive argument, using
  Proposition~\ref{prop:divisor-mult} at the root,
  Lemma~\ref{lem:rolle-cycle} at the recursive step, and
  Proposition~\ref{prop:cycle-dist-vs-mult} at the leafs. It is
  necessary only to note that for a zero-dimensional leaf $[p]$ we
  have $\mult_p^{\gamma_p}[p]=1$ by definition.
\end{proof}

\begin{proof}[Proof of Corollary~\ref{cor:mult-estimate}]
  The estimate follows in a straightforward manner from
  Theorem~\ref{thm:mult-forest}, by noting that the multiplicity
  of a cycle at a point is bounded by the cycle's degree.
\end{proof}

\section{Preliminaries on $(\C^*)^n$ and convex geometry}
\label{sec:convex}

In this section we consider the case $M=(\C^*)^n$. Given a Laurent
polynomial $P\in R$, we define its support $\supp P\subset\Z^n$ to be
the set of exponents appearing with non-zero coefficients in $P$. For
any set $A\subset\R^n$ we denote by $\conv(A)$ the convex hull of $A$.
The convex hull of a finite subset $A\subset\Z^n$ is called an
integral polytope. Finally, define the Newton polytope of $P$ to be
$\Delta(P):=\conv(\supp P)$.

For each set $A\subset\R^n$ we denote by $L_A$ the linear space of
polynomials whose support is contained in $A$.

\subsection{Toric classes}

\begin{Def}
  A \emph{toric $k$-class} $T$ is a symmetric map assigning a
  non-negative number $T(L_{A_1},\ldots,L_{A_k})\in\Z_{\ge0}$ for each
  collection of finite sets $A_1,\ldots,A_k\subset\Z^n$. We identify
  the toric $0$-classes with $\Z_{\ge0}$.

  If $\Gamma$ is a $k$-cycle, we define the corresponding toric class
  $\tc(\Gamma)$ by associating to each tuple $L_{A_1},\ldots,L_{A_k}$
  the number
  \begin{equation}
    \tc(\Gamma)(L_{A_1},\ldots,L_{A_k}) := \deg (\Gamma\cdot V(P_1)\cdots V(P_k))
  \end{equation}
  where $P_i$ is a generic element of $L_{A_i}$. This definition (for
  fixed $\Gamma$) was used in \cite{kk:intersection-theory} to develop
  a type of birationally equivalent intersection theory. In particular
  it is shown in \cite{kk:intersection-theory} that the number above
  is well defined.
\end{Def}

It is easy to see that the map $\Gamma\to\tc(\Gamma)$ is linear. We
define the (partial) order relation $\le$ on the space of toric
$k$-classes to be the pointwise ordering. If $A\subset\Z^n$ is finite
set and $T$ is a toric $k$-class, we define the product $T\cdot L_A$
to be the toric $k-1$-class defined by
\begin{equation}
  (T\cdot L_A)(L_{A_1},\ldots,L_{A_{k-1}}) := T(L_A,L_{A_1},\ldots,L_{A_{k-1}})
\end{equation}
This product is linear and respects the order relation.

\begin{Rem}\label{rem:L_A-identification}
  We sometimes identify $L_A$ with the toric $(n-1)$-class
  $[M]\cdot L_A$.
\end{Rem}

By definition, if $\Gamma$ is a cycle and $f\in L_A$
is such that $V(f)$ meets $\Gamma$ properly then
\begin{equation}
  \tc(\Gamma\cdot V(f)) \le \tc(\Gamma)\cdot L_A
\end{equation}

If $V\subset M$ is an irreducible variety, we define its toric Hilbert
function to be $\hf(V,A):=\dim (L_A)\rest V$ for any finite
$A\subset\Z^n$. In other words, $\hf(V,A)$ denotes the dimension of
the linear space spanned by the restrictions functions from $L_A$ to
$V$. The following proposition gives an upper bound for the toric
Hilbert function. The proof is analogous to that of \cite{bertrand:hf}
(suggested by Kollar). See also \cite{khovanskii:hf} for a different
approach.

\begin{Prop}\label{prop:toric-hf-bound}
  Let $\Delta_x\cap\Z\subset A$. If $V\subset M$ is an irreducible
  variety of dimension $k$ then
  \begin{equation}
    \hf(V,A) \le \tc(V) (L_A)^k + k
  \end{equation}
\end{Prop}
\begin{proof}
  Denote $\hf:=\hf(V,A)$. Consider the map $\phi:V\to\C P^{\hf-1}$
  whose projective coordinates are given by $\hf$ linearly independent
  functions from $(L_A)\rest V$. By assumption, $\phi$ is injective
  (since $A$ contains the constant $1$ as well as the coordinate
  functions). Denote by $W$ the Zariski closure of $\phi(V)$. Then $W$
  is irreducible and $\dim W=\dim V=k$.

  A generic projective space $L$ of codimension $k$ in $\C P^{\hf-1}$
  meets $W$ at points of $\phi(V)$. Since the pullbacks of the linear
  forms on $\C P^{\hf-1}$ to $V$ correspond bijectively to elements of $L_A$,
  we have $\deg W = \tc(V)(L_A)^k$.

  By definition, $W$ is not contained in any proper projective
  subspace of $\C P^{\hf-1}$. The claim now follows from the following
  classical fact: for any irreducible variety $W\subset \C P^{\hf-1}$
  which is not contained in a proper projective subspace,
  $\hf\le\deg W+\dim W$.

  The fact can be proved as follows. First, if $\dim W>1$ then passing
  to a generic projective hyperplane section does not affect the
  inequality (both sides are decreased by $1$) and preserves the
  irreducibility of $W$. Thus it suffices to prove the claim when
  $\dim W=1$. In this case, one can certainly choose a projective
  hyperplane meeting (any) $\hf-1$ points in $W$. Since this
  hyperplane does not contain $W$ by assumption, it follows that
  indeed $\deg W\ge\hf-\dim W$.
\end{proof}

\subsection{Mixed volume and the Bernstein-Kushnirenko theorem}
\label{sec:bk}

Recall that for $n$ convex bodies $\Delta_1,\ldots,\Delta_n$ in
$\R^n$, their \emph{mixed volume} is defined to be
\begin{equation}
  V(\Delta_1,\ldots,\Delta_n) = \pd{^n}{\lambda_1\cdots\partial \lambda_n} \vol(\lambda_1\Delta_1+\cdots+\lambda_n\Delta_n)
  \rest{\lambda_1=\cdots=\lambda_n=0}.
\end{equation}
The mixed volume is symmetric and multilinear, and generates
the volume function in the sense that $V(\Delta,\ldots,\Delta) = \vol(\Delta)$.
In fact, these properties completely determine the mixed volume
function.

The following result, known as the Bernstein-Kushnirenko theorem,
related the number of solutions of a generic system of polynomial
equations with prescribed supports to the mixed volume of their
Newton polytopes.

\begin{Thm}[\protect{\cite{kushnirenko:bk,bernstein:bk}}]
  Let $A_1,\ldots,A_n\subset\Z^n$ be finite sets. Then for generic
  $P_i\in L_{A_i}$, the system of equations $P_1=\cdots=P_n=0$ admits
  exactly $\mu$ solutions in $(\C^*)^n$, where
  \begin{equation} \label{eq:bk}
    \mu = n! V(\conv(A_1),\ldots,\conv(A_n)).
  \end{equation}
  In other words,
  \begin{equation}
    L_{A_1}\cdots L_{A_n} = n! V(\conv(A_1),\ldots,\conv(A_n)).
  \end{equation}
\end{Thm}

\begin{Rem}\label{rem:co-ideals}
  If the integral polytopes $\Delta_1,\ldots,\Delta_n\subset\Z_{\ge0}^n$ are co-ideals,
  then the Bernstein-Kushnirenko in fact gives an estimate for
  the number of solutions $\mu$ of the corresponding generic
  system of equation \emph{in $\C^n$}. Indeed, the corresponing
  spaces $L_{\Delta_i}$ are closed under the translation $\vec x\to \vec x+\vec a$, and one
  can therefore assume that the solutions of the generic system of
  equations fall outside of the $x_i$-axes, i.e. the number of
  solutions in $(\C^*)^n$ is the same as in $\C^n$.

  All of our results concerning the torus $(\C^*)^n$ could therefore
  be extended to the case $\C^n$ under the further assumption that the
  integral polytopes under consideration are co-ideals.
\end{Rem}

Let $\Delta_x$ denote the standard simplex in the $x$-variables in
$\R^n$. For any convex body $\Delta$ and $j=0,\ldots,n$ we define the
$j$-th (simplicial) \emph{quermassintegral} as
\begin{equation}
  \qmi_j(\Delta) = V(\underbrace{\Delta,\ldots,\Delta}_{n-j\text{ times}},
    \underbrace{\Delta_x,\ldots,\Delta_x}_{j\text{ times}}).
\end{equation}
We note that it is customary to use the Euclidean ball in place of the
standard simplex $\Delta_x$, but for our purposes the simplicial
normalization is more convenient.

\subsection{Two elementary lemmas on volumes and integral volumes}

Let $\Pi_n:=[-1,1]^n$ denote the unit cube in $\R^n$. In general we
say that $\Pi$ is an integral box if it is of the form
$[a_1,b_1]\times\cdots\times[a_n,b_n]$ where $a_i,b_i\in\Z$.

For any body $\Delta\subset\R^n$ Denote by $\vol(\Delta)$ the volume
of $\Delta$ and by $\ivol(\Delta)$ the number of integral points in
$\Delta$.

\begin{Lem}\label{lem:vol-v-ivol-helper}
  For any convex body $\Delta$, we have
  $\ivol(\Delta+\Pi_n)\ge\vol(\Delta)$. Moreover, if $\Pi$ is any
  integral box then
  $\ivol(\Pi\cap(\Delta+\Pi_n))\ge\vol(\Pi\cap\Delta)$.
\end{Lem}
\begin{proof}
  Let $A$ denote the union of all cubes of the form $z+[0,1]^n, z\in\Z^n$
  that meet $\Delta$. Then one easily checks that
  \begin{enumerate}
  \item $\vol(\Pi\cap A)\ge\vol(\Pi\cap\Delta)$.
  \item $\ivol(\Pi\cap A)\ge\vol(\Pi\cap A)$.
  \item $\Pi\cap A\subset \Pi\cap(\Delta+\Pi_n)$.
  \end{enumerate}
  The claim follows immediately.
\end{proof}

\begin{Lem}\label{lem:vol-v-ivol}
  Let $\Delta\subset\R^n$ denote a convex polytope such that
  $n\Pi_n\subset\Delta$. Then
  $\ivol(\Delta)\ge \frac{1}{4} \vol(\Delta)$. Moreover, if $\Pi$ is
  any integral box then
  $\ivol(\Pi\cap\Delta)\ge \frac{1}{4} \vol(\Pi\cap\Delta)$.
\end{Lem}
\begin{proof}
  Let $\Delta':=(1-1/n)\Delta$. Then $\Delta'+\Pi_n\subset\Delta$ and
  \begin{equation}
    \vol(\Pi\cap \Delta')\ge \vol((1-\tfrac{1}{n})(\Pi\cap\Delta))\ge \frac{1}{4} \vol(\Pi\cap\Delta)
  \end{equation}
  The statement now follows from Lemma~\ref{lem:vol-v-ivol-helper}
  applied to $\Delta'$.
\end{proof}

\section{Multiplicity estimates in the toric case}
\label{sec:toric-estimates}

In this section we consider the ambient space $M=(\C^*)^n$. We let
$\xi$ denote the vector field
\begin{equation}
  \xi = \sum_{i=1}^n \xi_i(x) \pd{}{x_i}, \qquad \xi_i\in R
\end{equation}
We define new Newton polytope of $\xi$, denoted $\Delta_\xi$, in the
same way we do for polynomials, where to each term
$x^\alpha\pd{}{x_i}$ we associate the same point in $\Z^n$ as we do
for $x^\alpha/x^i$. We suppose for simplicity that $\Delta_\xi$
contains the origin (there is no loss of generality, since all
problems considered in this section are invariant under multiplication
of $\xi$ by a monomial).

For any polytope $\Delta\subset\R^n$ we define $\deg_\Pi \Delta$ to be
the minimal $d\in\N$ such that $\Delta\subset d\Pi_n$. We let
$\deg_\Pi P:=\deg_\Pi\Delta(P)$.

\subsection{The key lemma in the toric case}

Let $\Delta$ be an integral polytope with $n\Pi_n\subset\Delta$. We
denote $\Delta_d:=\Delta\cap d\Pi_n$.

We begin with a simple lemma analogous to~\ref{lem:hf-poly-finder}.

\begin{Lem}\label{lem:toric-hf-poly-finder}
  Let $d>2n$ and let $V$ be an irreducible variety of dimension $k$ with
  \begin{equation}
    \tc(V)(L_{\Delta_d})^k \le B^{-1}_n (L_{\Delta_d})^n\qquad B_n= 2^4n!
  \end{equation}
  Then $I_V$ contains a polynomial $P\in L_\Delta$ with $\pdeg P<d$.
\end{Lem}
\begin{proof}
  By Proposition~\ref{prop:toric-hf-bound},
  \begin{multline}
    \hf(V,\Delta_{d-1}) \le \tc(V)(L_{\Delta_{d-1}})^k + k 
    \le B_n^{-1} (L_{\Delta_d})^n + k =\\
    B_n^{-1} n! \vol(\Delta_d) + k \le (B_n^{-1} n! + k (2n)^{-n}) \vol(\Delta_d) \le 2^{-3} \vol(\Delta_d)
  \end{multline}
  and by Lemma~\ref{lem:vol-v-ivol},
  \begin{equation}
    \begin{aligned}
      \hf(M,\Delta_{d-1}) &= \ivol(\Delta_{d-1}) \ge \ivol((1-\tfrac{1}{d})\Delta_d)
      > \tfrac{1}{4}\vol((1-\tfrac{1}{d})\Delta_d) \\
      &\ge 2^{-3} \vol(\Delta_d).
    \end{aligned}
  \end{equation}
  It follows that some non-trivial element of $L_{\Delta_{d-1}}$ vanishes
  on $V$, as claimed.
\end{proof}

We are now ready to state the toric analog of the key
Lemma~\ref{lem:main}.

\begin{Lem} \label{lem:toric-main} Let $V\subset M$ be an irreducible
  variety of dimension $k$ and suppose that $V$ is not contained in a
  (non-trivial) $\xi$-invariant variety. Suppose that
  \begin{equation}
    B_n2^n n \Delta_\xi \subset \Delta
  \end{equation}
  and that $I_V$ intersects $L_\Delta$ nontrivially, and let $P$ be a
  non-zero polynomial with minimal $d:=\pdeg P$ in the intersection.
  Then
  \begin{equation}
    \mult^\xi_V P\le b_{n,\delta}, \qquad b_{n,\delta}=\cN(n,\delta,(2n)B_n2^nn\delta)
  \end{equation}
\end{Lem}
\begin{proof}
  The proof is analogous to the proof of Lemma~\ref{lem:main}.
  Henceforth we assume that $d>B_n 2^n n \delta$. Otherwise, the
  statement of the lemma follows from Theorem~\ref{thm:mult-morse},
  applied to the polynomial $(x_1\cdots x_n)^d P$. Alternatively one
  can apply the toric estimate from \cite{mult-morse} directly to $P$
  to obtain a slightly better estimate.
  
  The claim used in the proof is replaced by
  \begin{Claim}
    Let $W\subset M$ be an irreducible variety of dimension $l\ge k$
    with $V\subset W$. Then
    $\tc W\cdot(L_{\Delta_d})^l > A_n^{-1}(L_{\Delta_d})^n$.
  \end{Claim}
  \noindent This claim is proved in the same way, based on
  Lemma~\ref{lem:toric-hf-poly-finder}. In the definition of the
  sequence $\Gamma^k$ one replaces condition~(4) by
  \begin{enumerate}
  \item[4.] The toric class of $\Gamma^j$ is bounded, $\tc(\Gamma^j)\le (L_{2\Delta_d})^j$.
  \end{enumerate}
  Suppose that $\Gamma^j$ has been constructed,
  $\Gamma^j=\sum m^j_i [W^j_i]$. By assumption we have
  $V\subset W^j_i$, and by the preceding claim
  $\tc W_i^j \cdot(L_{\Delta_d})^{n-j} > B_n^{-1}(L_{\Delta_d})^n$. Using condition~(4),
  \begin{equation}
    m_i^j B_n^{-1} (L_{\Delta_d})^n < \tc(\Gamma^j)\cdot (L_{\Delta_d})^{n-j} \le
    (L_{2\Delta_d})^j (L_{\Delta_d})^{n-j} = 2^j (L_{\Delta_d})^n
  \end{equation}
  where we used the Bernstein-Kushnirenko theorem and that fact that
  $\Delta_d,2\Delta_d$ are integral polytopes for the last step.
  Therefore $m^j_i< B_n 2^j$. By condition~(3) and
  Proposition~\ref{prop:mult-transverse} we have
  \begin{equation} \label{eq:toric-p-wji-mult}
    \begin{aligned}
      \mult^\xi_{W^j_i} P &\le (j-1)\mult_{W^j_i} \Gamma^j + \mult^\xi_{W^j_i}\Gamma^j \\
       &= (j-1)m^j_i + m^j_i < B_n 2^j j
    \end{aligned}
  \end{equation}
  If $\mult^\xi_{W^j_i}P=\mult^\xi_VP$ then~\eqref{eq:toric-p-wji-mult}
  proves the claim of the lemma. Otherwise, let $Q^j_i$ denote the
  first $\xi$-derivative of $P$ which does not vanish on $W^j_i$ (but
  still vanishes on $V$). By~\eqref{eq:toric-p-wji-mult} we have
  \begin{equation}\label{eq:toric-qij-deg}
    \Delta(Q^j_i) \subset \Delta_d +(B_n 2^j j)\Delta_\xi \subset 2 \Delta_d
  \end{equation}

  The rest of the proof is entirely analogous to the proof of
  Lemma~\ref{lem:main}. We leave the verification of the details to
  the reader.
\end{proof}

\subsection{The multiplicity forest}

The following result is the toric analog of Theorem~\ref{thm:mult-forest}.

\begin{Thm}\label{thm:toric-mult-forest}
  Let $P\in R$. There exists a cycle forest $T_P$ with the same
  properties as in Theorem~\ref{thm:mult-forest}, with condition~(3)
  replaced by
  \begin{enumerate}
  \item[3.] The toric class of $T_P^k$ is bounded,
    \begin{equation} \label{eq:toric-tpk-deg}
      \tc(T_P^k) \le (L_{\Delta+\Delta_{n,\xi}})^k  \qquad \Delta_{n,\xi} = (B_n2^n n+b_{n,\delta})\Delta_\xi
    \end{equation}
  \end{enumerate}
\end{Thm}
\begin{proof}
  Let $\Delta':=\Delta+B_n2^n n\Delta_\xi$. Then
  Lemma~\ref{lem:toric-main} applies to $\Delta'$ and the rest of the
  proof is analogous to the proof of Theorem~\ref{thm:mult-forest}. We
  leave the details to the reader.
\end{proof}

\begin{Cor}\label{cor:toric-mult-estimate}
  Let $p\in M$ and $\gamma_p$ a smooth analytic trajectory of $\xi$
  through $p$. Suppose that $\gamma_p$ satisfies the D-property with
  constant $\chi$. Then for any $P\in R$ with $\Delta(P)=\Delta$,
  \begin{equation}
    \mult_p^{\gamma_p} P \le n! \vol(\Delta+\Delta_{n,\xi}) + n!(2+\chi)\qmi_1(\Delta+\Delta_{n,\xi})
  \end{equation}
\end{Cor}
\begin{proof}
  The estimate follows in a straightforward manner from
  Theorem~\ref{thm:toric-mult-forest}. Namely, for any node in the
  multiplicity forest with toric $k$-class $\cT$, we estimate its
  degree by $\cT\cdot (L_{\Delta_x})^k$. The statement follows
  from~\eqref{eq:toric-tpk-deg} using the identity
  \begin{equation}
    (L_{\Delta+\Delta_{n,\xi}})^n = n! \vol(\Delta+\Delta_{n,\xi})
  \end{equation}
  and the inequality
  \begin{equation}
    (L_{\Delta+\Delta_{n,\xi}})^{n-k}\cdot (L_{\Delta_x})^k \le n! W_1(\Delta+\Delta_{n,\xi})
    \qquad \text{for }k\ge1
  \end{equation}
  both of which follow from the Bernstein-Kushnirenko theorem.
\end{proof}

\subsection{Recovering the classical multiplicity estimates}
\label{sec:recover-nesterenko}

In subsection we show how the toric estimates presented in this
section imply the various known multiplicity estimates for mixed
degrees as a special case. We thus consider the ambient space
$\C\times\C^n$, where we denote the first coordinate by $z$, thought
of as the time variable, and the remaining coordinates by $x$, thought
of as dependent variables. As usual we denote by $R$ the corresponding
polynomial ring.

Consider a vector field $\xi$ of the form
\begin{equation}
  \xi = t(z) \pd{}{z}+\sum_{i=1}^n \xi_i(z,x) \pd{}{x_i}.
\end{equation}
We denote by $\Delta_z,\Delta_x$ the standard simplices in the $z$ and
$x$ variables, respectively. We fix a trajectory
$\gamma=(z,f_1(z),\ldots,f_n(z))$ of $\xi$, with $f_1,\ldots,f_n$
holomorphic in some domain $U\subset\C$.

\subsubsection{Estimate for a single point}

We begin with Nesterenko's classical estimate for the multiplicity at
a single point (see~\secref{sec:nest-d-prop},
Theorem~\ref{thm:nest-main}).

\begin{Thm*}
  Let $p\in U$ and suppose that $\gamma$ has the D-property at $p$.
  Let $P\in R$ be a polynomial with $P\rest\gamma\not\equiv0$, and denote
  \begin{equation}
    d_x := \max(\deg_x P, 1) \qquad d_z := \max(\deg_z P, 1)
  \end{equation}
  Then
  \begin{equation}
    \mult_{z=p} P(z,f_1(z),\ldots,f_n(z)) \le \alpha_\gamma d_z d_x^n 
  \end{equation}
  where $\alpha_\gamma$ is a constant dependending only on $\gamma$.
\end{Thm*}
\begin{proof}
  Let $\chi$ denote the D-property constant of $\gamma$ at $p$.
  We have $\Delta(P)\subset d_x\Delta_x+d_z\Delta_z$, and by Corollary~\ref{cor:toric-mult-estimate}
  we have
  \begin{align*}
    \mult_p^{\gamma_p} P &\le (n+1)! \vol(\Delta(P)+\Delta_{n,\xi}) + (n+1)!(2+\chi)\qmi_1(\Delta(P)+\Delta_{n,\xi}) \\
    &\le (n+1)! (3+\chi) (\deg_\Pi \Delta_{n,\xi}) \vol(d_x\Delta_x+d_z\Delta_z) \\
    &\le \alpha_\gamma d_z d_x^n
  \end{align*}
  for an appropriate constant $\alpha_\gamma$.
\end{proof}

\subsubsection{Estimate for multiple points}
\label{sec:several-pts-estimate}

We now consider the analogous result for the case involving multiple
points (see~\secref{sec:nest-d-prop},
Theorem~\ref{thm:nest-many-pts}).

\begin{Thm*}
  Let $p_1,\dots,p_q\in U$ and suppose that $\gamma$ has the D-property at $p_i$
  for every $i$. Let $P\in R$ be a polynomial with $P\rest\gamma\not\equiv0$, and denote
  \begin{equation}
    d_x := \max(\deg_x P, 1) \qquad d_z := \max(\deg_z P, 1)
  \end{equation}
  Then
  \begin{equation}
    \sum_{i=1}^q \mult_{z=p} P(z,f_1(z),\ldots,f_n(z)) \le \beta_\gamma (d_z+q) d_x^n 
  \end{equation}
  where $\beta_\gamma$ is a constant dependending only on $\gamma$.
\end{Thm*}
\begin{proof}
  Note that the D-property holds automatically with constant $1$ for
  regular points of $\xi$, i.e. for any $p_i$ except for the (finitely
  many) roots of $t(z)$. We may thus assume that the D-property holds
  with a uniform constant $\chi$ independent of the choice of the
  points $p_i$.

  Applying Theorem~\ref{thm:toric-mult-forest} we have a multiplicity
  forest for $P$ with the estimate
  \begin{equation} \label{eq:nest-tpk-est}
    \tc(T_P^k) \le (L_{\beta_1(d_x\Delta_x+d_z\Delta_z)})^k
  \end{equation}
  for some constant $\beta_1$ (depending only on $\xi$). We write $T^k_P:=\tilde T^k_P+\hat T^k_P$
  where $\tilde T^k_P$ consists of those component of $T^k_P$ which
  are strictly contained in a hyperplane $z=\const$, and $\hat T^k_P$
  the rest.

  Each component of $\tilde T^k_P$ can contain at most one point
  $p_i$. Thus a simple computation using~\eqref{eq:nest-tpk-est} gives
  \begin{equation} \label{eq:tpk-tilde}
    \sum_{i=1}^q \mult_{p_i} \tilde T^k_P \le \deg \tilde T^k_P \le \beta_2 d_z d_x^n.
  \end{equation}
  On the other hand, since the components of $\hat T^k_P$ are not
  contained in $z=\const$, we may estimate their multiplicity from
  above at any point $p_i$ by intersecting with the hyperplane $z=z(p_i)$
  and $n-k$ additional generic hyperplanes. Thus
  \begin{equation} \label{eq:tpk-hat}
    \mult_{p_i} \hat T^k_P \le
    V(\underbrace{\beta_1(d_x\Delta_x+d_z\Delta_z)}_{k\text{ times}},\Delta_z,\underbrace{\Delta_z+\Delta_x}_{n-k\text{ times}})
    \le \beta_3 d_x^n
  \end{equation}
  where we expanded the middle expression by multilinearity and used the
  fact that if the $\Delta_z$ term appears twice then then mixed volume
  is zero (since $\Delta_z$ is one-dimensional).

  Finally, using Theorem~\ref{thm:toric-mult-forest}
  and~\eqref{eq:tpk-tilde},~\eqref{eq:tpk-hat} the conclusion of the
  theorem easily follows.
\end{proof}

\subsubsection{The case of a trajectory satisfying algebraic relations}
\label{sec:algebraic-relations}

Let $Z$ denote the Zariski closure of $\gamma$. We suppose now that
$Z\neq M$. In this case $\gamma$ certainly does not satisfy the
D-property, because the ideal $I_Z$ is invariant and any function
$P\in I_Z$ vanishes identically on $\gamma$. One can avoid such
``trivial'' counterexamples by considering $\xi$ as a vector field in
the \emph{ambient space} $Z$. Indeed, since $Z$ is the Zariski closure
of the irreducible invariant set $\gamma$, it follows that $Z$ is
itself irreducible and invariant, and $\xi$ induces a derivation of
the ring $\cO_Z$. The D-property in this context states that any
non-zero $\xi$-ivariant prime ideal $J\subset\cO_Z$ (i.e., any
$\xi$-invariant prime ideal $J\subset\cO_M$ strictly containing $I_Z$)
contains a function $F$ with $\mult_p^{\gamma_p}F\le\chi$.

A result in this context was stated in \cite{nesterenko:modular} with
a small technical mistake (the proof was only sketched). Dolgalev
\cite{dolgalev:mult} gave a corrected formulation and a full proof. In
this subsection we establish a strenghening of Dolgalev's results.

Our proof essentially extends to the present context with little
changes. One can repeat all considerations within the ambient space
$Z$. Specifically, in Remark~\ref{rem:L_A-identification} we now
identify $L_A$ with the toric class $\tc(Z)\cdot L_A$.
Lemma~\ref{lem:toric-main} and Theorem~\ref{thm:toric-mult-forest} as
well as their proofs extend literally. The only exception is
Lemma~\ref{lem:toric-hf-poly-finder}, where a lower bound for the
toric Hilbert function of the ambient space, in our case $(\C^*)^n$,
was explicitly used. This bound must be replaced by an appropriate
lower bound for $\hf(Z,\Delta_d)$.

We show how this can be carried out in the context of this subsection,
namely for $\Delta$ of the form $d_x\Delta_x+d_z\Delta_z$. Denote the
dimension of $Z$ by $m$.

Let $\pi_x:\C\times\C^n\to\C^n$ denote the projection to the $x$
variables. We distinguish between two cases for $Z$:
\begin{enumerate}
\item[A.] We have $\dim\pi_x(Z)=\dim Z-1$. In this cas $Z=\C\times Z_x$
  where $Z_x\subset\C^n$ is an irreducible variaty.
\item[B.] We have $\dim \pi_x(Z)=\dim Z$.
\end{enumerate}

\begin{Lem}\label{lem:hf-lower-bound}
  Denote $d'_z=\min(d_z,d),d'_x=\min(d_x,d)$ (these are simply the
  $z$ and $x$ sizes of polytope $\Delta_d=\Delta\cap\Pi_d$).
  \begin{equation}
    \hf(Z,\Delta_d) \ge (m!)^{-1} \cdot
    \begin{cases}
      d'_z (d'_x)^{m-1} & \text{in case A} \\
      \max(d'_z (d'_x)^{m-1} ,(d'_x)^{m})  & \text{in case B}
    \end{cases}
  \end{equation}
\end{Lem}
\begin{proof}
  In case A, after a generic linear change in the $x$ variables we may
  assume that the projection
  \begin{equation*}
    \pi:\C\times\C^n\to\C\times\C^{m-1} \qquad \pi(z,x_1,\ldots,x_n)=(z,x_1,\ldots,x_{m-1})
  \end{equation*}
  is rational dominant. Thus $I_Z\cap\C[z,x_1,\ldots,x_{m-1}]=\{0\}$,
  and hence $\hf(Z,\Delta_d)$ is at least the dimension of the space
  of polynomials in $z,x_1,\ldots,x_{m-1}$ with $\deg_z\le d'_z$ and
  $\deg_x\le d'_x$. The result follows by simple arithmetic.

  In case B we argue similarly. In this case, after a generic linear
  change in the $x$ variables we may assume that the two projections
  \begin{align*}
    \pi_1:\C&\times\C^n\to\C\times\C^{m-1} & \qquad \pi(z,x_1,\ldots,x_n)&=(z,x_1,\ldots,x_{m-1}) \\
    \pi_2:\C&\times\C^n\to\C^m & \qquad \pi(z,x_1,\ldots,x_n)&=(x_1,\ldots,x_{m-1},x_m)
  \end{align*}
  are both rational dominant. We can thus use either set of variables
  to produce a lower bound for $\hf(Z,\Delta_d)$, and the result
  follows as above.
\end{proof}

The ideal $I(Z)$ is generated by polynomials bounded by some degree
$D$, and $Z$ is a component of a generic combination of these
generators. In case A, one can further assume that the generators are
independent of $z$. Thus
\begin{equation} \label{eq:z-tc}
  \begin{aligned}
    \tc(Z) &\le (L_{D\Delta_x})^{n-m} & \qquad \text{in case A} \\
    \tc(Z) &\le (L_{D(\Delta_x+\Delta_z)})^{n-m} & \qquad \text{in case B}
  \end{aligned}
\end{equation}

\begin{Lem}
  Let $d>2n$ and let $V\subset Z$ be an irreducible variety of dimension $k$ with
  \begin{equation}
    \tc(V)(L_{\Delta_d})^k \le C^{-1}_n \tc(Z)(L_{\Delta_d})^m \qquad C_{n,Z}= \const(n,Z)
  \end{equation}
  Then $I_V\setminus I_Z$ contains a polynomial $P\in L_\Delta$ with $\pdeg P<d$.
\end{Lem}
\begin{proof}
  By Proposition~\ref{prop:toric-hf-bound},
  \begin{equation*}
    \hf(V,\Delta_{d-1}) \le \tc(V)(L_{\Delta_{d-1}})^k + k 
    \le C^{-1}_n \tc(Z)(L_{\Delta_d})^m + k
  \end{equation*}
  and using~\eqref{eq:z-tc} and the Bernstein-Kushnirenko theorem,
  \begin{equation*}
    \hf(V,\Delta_{d-1}) \le k+
    \begin{cases}
      D^{n-m} V(\underbrace{\Delta_d}_{m\text{ times}},\underbrace{\Delta_x}_{n-m\text{ times}}) & \text{in case A} \\
      D^{n-m} V(\underbrace{\Delta_d}_{m\text{ times}},\underbrace{\Delta_x+\Delta_z}_{n-m\text{ times}}) & \text{in case B}
    \end{cases}
  \end{equation*}
  One easily checks that in both cases, this agrees (up to a constant
  dependending only on $n,D$) with the lower bound for
  $\hf(Z,\Delta_d)$ given in Lemma~\ref{lem:hf-lower-bound}. The proof
  can be concluded exactly as in the proof of
  Lemma~\ref{lem:toric-hf-poly-finder}.
\end{proof}

We leave the verification of the details of
Lemma~\ref{lem:toric-main}, Theorem~\ref{thm:toric-mult-forest} and
their corollaries in this context for the reader. In a manner
analogous to the result of~\secref{sec:several-pts-estimate}, one
obtains the following theorem, which extends the two main theorems of
\cite{dolgalev:mult}.

\begin{Thm}
  Let $p_1,\dots,p_q\in U$ and suppose that $\gamma$ has the D-property at $p_i$
  for every $i$. Recall that $m$ denotes the dimension of the Zariski
  closure of $\gamma$. Let $P\in R$ be a polynomial with $P\rest\gamma\not\equiv0$, and denote
  \begin{equation}
    d_x := \max(\deg_x P, 1) \qquad d_z := \max(\deg_z P, 1)
  \end{equation}
  Then
  \begin{equation}
    \sum_{i=1}^q \mult_{z=p} P(z,f_1(z),\ldots,f_n(z)) \le
    \begin{cases}
      \beta_\gamma (d_z+q) d_x^{m-1}      & \text{in case A} \\
      \beta_\gamma (d_z+d_x+q) d_x^{m-1} & \text{in case B} \\
    \end{cases}
  \end{equation}
  where $\beta_\gamma$ is a constant dependending only on $\gamma$.
\end{Thm}

\section{Concluding remarks}
\label{sec:conclusion}

In this paper we have defined the notion of the multiplicity of a
cycle, specifically along a \emph{smooth analytic curve}. We have
restricted our attention to this case in order to simplify our
presentation, and because this is the case which is needed for the
classical multiplicity estimates that we have sought to strengthen.
However, the simple algebraic nature of
Definition~\ref{def:cycle-curve-mult} easily lends itself to
generalization.

One interesting direction for such generalization is the study of
foliations defined by several commuting vector fields (in place of the
single vector field considered in this paper). Let
$\xi_1,\ldots,\xi_m$ denote the germs of $m$ commuting polynomial
vector fields in $\C^n$, and let $\cL_p$ denote the germ of a smooth
analytic leaf at the point $p$. One can consider multiplicity
estimates of the following types:
\begin{itemize}
\item For a polynomial $P$, one may ask about $\ord_p P\rest{\cL_p}$.
\item More generally, for polynomials $P_1,\ldots,P_m$ one may ask
  about the multiplicity of their common root,
  $\mult_p (P_1\rest{\cL_p},\ldots,P_m\rest{\cL_p})$.
\end{itemize}
Multiplicity estimates in the multi-dimensional setting have been used
in transcendental number theory, for instance in
\cite{mw:groups1,mw:groups2,philippon:groups}. They have also been
studied in a more geometric context, for instance in
\cite{gk:noetherian}.

Definition~\ref{def:cycle-curve-mult} extends to the multi-dimensional
setting without change --- one should simply replace the ideal of
definition of the curve $\gamma_p$ by the ideal of definition of
$\cL_p$. Moreover, a generalization of the Rolle-type
Lemma~\ref{lem:rolle-cycle} holds in this context as well. It appears
plausible that much of the theory developed in this paper could be
carried out for the multi-dimensional setting.

Finally, we would like to mention that multiplicity estimates have
also been considered in the related context of functions satisfying
Mahler-type functional equations, similarly leading to applications in
transcendental number theory. Moreover, Nesterenko's methods have been
successfully applied in this context (see \cite{nishioka:mahler} for
example). It would be interesting to see whether the ideas developed
in this paper could similarly be applied in this context.

\appendix
\section{Appendix: A review on Multiplicity Estimates and the D-property}

In this appendix we present a brief summary of some multiplicity
estimates that have been studied in the literature, as well as a
somewhat more detailed account of Nesterenko's D-property and his
approach to multiplicity estimates.

\subsection{Historical review}
\label{sec:hist-review}

Multiplicity estimates have been considered by authors in various
areas of mathematics. We list some key contributions below, making
no attempt at a comprehensive review.

Multiplicity estimates have been extensively used in trancendental
number theory, starting with the work of Siegel~\cite{siegel} and
Shidlovskii~\cite{shidlovskii} on the class of E-functions.
Nesterekno, motivated by the study of E-functions, introduced
elimination theoretic ideas for the study of multiplicity estimates
for general linear systems of differential equations in
\cite{nesterenko:e-funcs,nesterenko:mult-lin1}. Further results in the
linear context have been studied by Nesterenko
\cite{nesterenko:mult-lin2,nesterenko:mult-lin3} and Nguen
\cite{nguen} using similar methods, and by Bertrand and Beukers
\cite{bb:mult} using a different approach.

The case of non-linear non-singular systems was considered by
Brownawell and Masser in \cite{bm:mult-I,bm:mult-II} and by Brownawell
in \cite{brownawell:zero-ord,brownawell:zero-est}. These results were
consequently improved to an essentially optimal result by Nesterenko
in \cite{nesterenko:mult-nonlinear}. The corresponding result for
singular systems was established by Nesterenko in
\cite{nesterenko:modular} (see also \cite{nesterenko:modular-book}),
and somehwat generalized by Dolgalev in \cite{dolgalev:mult}. An
alternative approach was given by Zorin in \cite{zorin}. For the key
example of the Ramanujan functions, similar estimates were obtained
by Philippon in \cite{philippon:k-funcs}.

We mention also that numerous important results have been obtained in
the more refined context of invariant vector fields on commutative
group varieties, for instance by Masser and W\"ustholz
\cite{mw:groups1,mw:groups2} and Philippon \cite{philippon:groups}
(see \cite{masser:zero-est-survey} for a survey).

In control theory, Risler \cite{risler:nonholonomy} showed how
multiplicity estimates could be used in the study of nonholonomic
systems, and carried out the program in the planar case, giving a
bound for the degree of non-holonomy of a polynomial system. Gabrielov
and Risler \cite{gr:mult-c3} established a similar result in dimension
3. Multiplicity estimates in arbitrary dimension were given by
Gabrielov in \cite{gabrielov:mult-old} and significantly improved in
\cite{gabrielov:mult}, giving for the first time an estimate
exhibiting simple exponential growth with respect to the dimension.

In the theory of dynamical systems multiplicity estimates have been
studied with the motivation of obtaining bounds on the bifurcation of
limit cycles in perturbations of Hamiltonian systems, for instance in
the work of Novikov and Yakovenko \cite{ny:chains} and Moura
\cite{moura:mult}. Results about bifurcation of zeros in analytic
families have been established by Yomdin in \cite{yomdin:oscillation}
with the help of Gabrielov's multiplicity estimate.

\subsection{The D-property and multiplicity estimates}
\label{sec:nest-d-prop}
  
Nesterenko has established that the D-property holds in two main cases
(each with different applications in trancendental number theory):
\begin{enumerate}
\item In his study of E-functions, Nesterenko \cite{nesterenko:galois}
  has used differential Galois theory to show that if the
  system~\eqref{eq:diff-system} is linear and $f$ is a completely
  transcendental solution (i.e., $f$ does not satisfy any nontrivial
  polynomial relation over $\C(z)$) then $f$ automatically satisfies
  the D-property whenever it is holomorphic, with some suitable
  constant $\chi$.

  More specifically, it is shown in \cite{nesterenko:galois} that
  if~\eqref{eq:diff-system} admits at least one completely
  transcendental solution then there exist \emph{finitely many}
  $\xi$-invariant proper varieties that are \emph{maximal} with
  respect to inclusion. The result easily follows since the graph of
  $f$ is not contained in any of these varieties, and therefore has
  finite order contact with them.
\item In his celebrated work on the algebraic independence of
  $\pi,e^\pi$ Nesterenko \cite{nesterenko:modular} considered the
  Ramanujan functions
  \begin{gather*}
    P(z) = 1-24\sum_{n=1}^\infty \sigma_1(n) z^n \\
    Q(Z) = 1+240\sum_{n=1}^\infty \sigma_3(n) z^n \\
    R(Z) = 1-504\sum_{n=1}^\infty \sigma_5(n) z^n
  \end{gather*}
  and the corresponding system of differential equations (due to
  Ramanujan \cite{ramanujan}),
  \begin{equation*}
    z \pd{P}{z} = \frac{1}{12}(P^2-Q), \quad
    z \pd{Q}{z} = \frac{1}{3}(PQ-R), \quad
    z \pd{R}{z} = \frac{1}{2}(PR-Q^2).
  \end{equation*}
  In \cite{nesterenko:modular} it is proved that the holomorphic
  solution $P,Q,R$ satisfies the D-property at $z=0$ with the constant
  $\chi=2$. Theorem~\ref{thm:nest-main} below, applied to the
  Ramanujan functions, was the main novel ingredient in
  \cite{nesterenko:modular}, giving general results on the
  transcendence properties of modular functions and in particular the
  algebraic independence of $\pi,e^\pi,\Gamma(1/4)$.
\end{enumerate}

Nesterenko has developed a powerful technique for proving multiplicity
estimates based on elimination theoretic methods
(see~\secref{sec:nest-approach} for a review). The $\xi$-invariant
prime ideals play a natural role as a basis for an induction over
dimension, and the D-property establishes the inductive hypothesis for
this basis. These ideas have been developed in
\cite{nesterenko:mult-lin1,nesterenko:mult-lin2,nesterenko:mult-lin3}
for linear systems, in \cite{nesterenko:mult-nonlinear} for non-linear
systems at nonsingular points, and culminated in
\cite{nesterenko:modular} with the following formulation, valid for
general non-linear systems with singular points.

\begin{Thm} [\protect{\cite[Theorem~3]{nesterenko:modular}}] \label{thm:nest-main}
  Let $z_0\in\C$ and suppose that $f(z)$ has the D-property at $z_0$.
  Let $P$ be a polynomial with $P(z,f(z))\not\equiv0$, and denote
  \begin{equation}
    d_x := \max(\deg_x P, 1) \qquad d_z := \max(\deg_z P, 1)
  \end{equation}
  Then
  \begin{equation}
    \mult_{z=p} P(z,f_1(z),\ldots,f_n(z)) \le \alpha_f d_z d_x^n 
  \end{equation}
  where $\alpha_f$ is a constant dependending only on $f$.
\end{Thm}

Once again, assuming that $f(z)$ satisfies no algebraic relations
over $\C(z)$, this result is optimal with respect to $d_z,d_x$ up
to the precise multiplicative constant.

The problem of estimating the sum of multiplicities over several
points was also considered by various authors
\cite{nesterenko:mult-lin3,bb:mult,nguen,nesterenko:mult-nonlinear,nesterenko:modular,dolgalev:mult}.
The following result is essentially optimal for the most general case
of nonlinear systems with
singularities\footnote{see~\secref{sec:algebraic-relations} for more
  refined formulations taking algebraic relations among the functions
  $f_1,\ldots,f_n$ into account}.
\begin{Thm}[\protect{\cite[Theorem~6]{nesterenko:modular}}] \label{thm:nest-many-pts}
  Let $z_1,\ldots,z_\nu\in\C$ and suppose that $f(z)$ has the
  D-property at $z_1,\ldots,z_\nu$. Let $P$ be a polynomial with
  $P(z,f(z))\not\equiv0$, and denote
  \begin{equation}
    d_x := \max(\deg_x P, 1) \qquad d_z := \max(\deg_z P, 1)
  \end{equation}
  Then
  \begin{equation}
    \sum_{i=1}^\nu \mult_{z=p_i} P(z,f_1(z),\ldots,f_n(z)) \le \alpha_f (d_z+\nu) d_x^n 
  \end{equation}
  where $\alpha_f$ is a constant dependending only on $f$ and the
  points $p_1,\ldots,p_\nu$.
\end{Thm}

\subsection{An overview of Nesterenko's approach to multiplicity estimates}
\label{sec:nest-approach}

Nesterenko's approach follows the same three-step paradigm outlined in
the begining of~\secref{sec:synopsis-pure}. We briefly sketch how each
of the steps is realized in his work, and point the reader to the
analogous results in the present paper for comparison. When possible
we have given references uniformly to \cite{nesterenko:modular} in
order to allow the reader to follow all references with fixed notations,
although many of the statements appear originally in Nesterenko's
earlier works. We refer the reader to \cite{nesterenko:modular} for
the original references.

We remark generally that in Nesterenko's approach one considers
projective ideals and proves the main results in this context. The
results in the original affine context later follow by a
projectivization argument. For simplicity we speak of \emph{unmixed
  ideals} below without qualification, and it is to be understood that
these ideals are taken to be projective ideals, and some technical
details related to projectivization are omitted.

\emph{Step 1.} The basis of Nesterenko's approach is a notion of a
multiplicity associated to an ideal, based on elimination-theoretic
ideas. Given a smooth analytic trajectory $\gamma_p$ at the point
$p\in M$, Nesterenko associates to each unmixed ideal $I$ the notion
of the \emph{order of $I$ along $\gamma_p$}, denoted
$\ord I(\gamma_p)$. The definition of this order is somewhat involved
(see \cite[Section~3]{nesterenko:modular}; cf.
Definition~\ref{def:cycle-curve-mult}). A rough idea (reinterpreted
from Nesterenko's formulation) for the construction is as follows. One
first considers the Chow form associated canonically to $I$. One then
associates to this Chow form a \emph{canonical system} of equations
for $I$, following a construction due to Chow and van der Waerden (see
\cite[3.2.C]{gkz}). The order of $I$ along $\gamma_p$ is defined to be
the minimal order of any of these canonical equations along
$\gamma_p$.

For this notion it is important that the ambient space is $\C P^n$, so
that unmixed ideals can be parametrized by Chow forms. In order to
study the case of mixed degrees, where the natural ambient space is
rather $\C P^1\times \C P^n$, Nesterenko considers it as a projective
$n$-space over the field $\C(z)$. One can then consider Chow forms
over the field $\C(z)$ and carry out the preceding construction in a
similar manner (although Nesterenko makes some technical
modifications).

In the case of $\C P^n$, each unmixed ideal has a naturally associated
\emph{degree}: its degree as a cycle in $\C P^n$. In the mixed case,
each unmixed ideal has two associated numbers: \emph{degree} and
\emph{height} (see \cite[Section~3]{nesterenko:modular}). They
essentially correspond to its two cohomological components in the
two-dimensional cohomology of $\C P^1\times\C P^n$ (with respect to
the Kunneth generators). In either case, we shall refer to these
numbers as the \emph{degrees} of $I$.

Nesterenko shows that the order of a principal ideal $\<P\>$ along
$\gamma_p$ is bounded in terms of the order of $P$ along $\gamma_p$
(see \cite[Proposition~1]{nesterenko:mult-lin3}; cf.
Proposition~\ref{prop:divisor-mult}), thus completing the first step.

\emph{Step 2.} To accomplish this step, Nesterenko uses two main
results. First, he proves (see \cite[Lemma~5.1]{nesterenko:modular};
cf. Lemma~\ref{lem:rolle-cycle}) that if $I$ is an unmixed ideal,
$P\in I$ and $Q=\xi P\not\in I$ then there is an unmixed ideal $J$
whose zeros coincide with the zeros of $I+\<Q\>$, such that:
\begin{itemize}
\item The degrees of $J$ are appropriately bounded in terms of the
  degrees of $I$ and $Q$.
\item $\ord I(\gamma_p)\le \ord J(\gamma_p)+C(I,Q)$ where $C(I,Q)$
  denotes a certain expression depending on the degrees of $I$ and
  $Q$.
\end{itemize}

Next, Nesterenko proves that for any unmixed ideal $I$ not containing
a proper $\xi$-invariant ideal, one can always choose a polynomial $P$
as above with the degrees of $P$ appropriately bounded in terms of the
degrees of $I$ (see \cite[Lemma~5.4]{nesterenko:modular}; cf.
Lemma~\ref{lem:main} and Lemma~\ref{lem:toric-main}). This lemma is
the deepest and most technical part of the proof. It has appeared in
various forms of increasing complexity in the work of Nesterenko: in
the linear case \cite{nesterenko:mult-lin3}, in the non-linear
pure-degree case \cite{nesterenko:mult-nonlinear} and finally in the
singular case with mixed degrees in \cite{nesterenko:modular}.

\begin{Rem}
  The lemma above, in addition to being the deepest part of the proof,
  also plays the dominant role in determining the size of the
  multiplicative constant $\alpha_f$ appearing in
  Theorem~\ref{thm:nest-main}. Namely, this constant grows
  doubly-exponentially with the dimension $n$. It is known from the
  work of Gabrielov \cite{gabrielov:mult} that the correct growth with
  respect to $n$ is singly-exponential, at least in the non-singular
  case.

  Our proof of this lemma follows a different approach, relying on the
  local nature of Lemma~\ref{lem:rolle-cycle} and, at one crucial
  moment, on Gabrielov's result (or the more refined form given in
  \cite{mult-morse}). This allows us to give constants growing
  singly-exponentially with $n$, and also to extend the result to the
  case of general Newton polytopes.
\end{Rem}

The combinations of these two lemmas allows one to construct from an
unmixed ideal $I$ a new unmixed ideal $J$ of smaller dimension, such
that the order of $\ord I(\gamma_p)$ is bounded in terms of
$\ord J(\gamma_p)$ and such the degrees of $J$ are bounded in terms of
the degrees of $I$ --- as long as $I$ is positive dimensional and
doesn't contain a proper $\xi$-invariant ideal. This concludes step 2.

\emph{Step 3.} Assume now that the D-property is satisfied with
constant $\chi$. The final step is accomplished by showing that if $J$
is a zero dimensional ideal or an unmixed ideal contained in a proper
$\xi$-invariant variety, then $\ord J(\gamma_p)$ is bounded by a
constant depending on $\chi$ (this is rougly
\cite[Lemma~5.3]{nesterenko:modular} formulated in the contrapositive;
cf. Proposition~\ref{prop:cycle-dist-vs-mult}).

\bibliographystyle{plain}
\bibliography{nrefs}

\end{document}